\title{On the phase transition in random simplicial complexes}
\author{Nathan Linial\thanks{Department of Computer Science, Hebrew University, Jerusalem 91904,
    Israel. e-mail: nati@cs.huji.ac.il~. Supported by ERC grant 339096 "High-dimensional combinatorics".} \and  {Yuval Peled\thanks{Department of Computer Science, Hebrew University, Jerusalem 91904,
    Israel. e-mail: yuvalp@cs.huji.ac.il~. YP is grateful to the Azrieli foundation for the award of an Azrieli Fellowship.}
}}

\date{\today}
\documentclass[11pt]{article}
\usepackage{amssymb,fullpage}
\usepackage{xspace}
\usepackage{latexsym}
\usepackage{times}
\usepackage{amsfonts}
\usepackage{amsmath}

\usepackage{mathrsfs}
\usepackage{bbm}

\usepackage{verbatim}

\usepackage{amsthm}
\usepackage{amsmath}
\usepackage{amssymb}
\usepackage{graphicx}
\usepackage{subfig}
\usepackage{epstopdf}
\usepackage{titling}

\usepackage{enumerate}

\usepackage{algpseudocode}
\usepackage{algorithm}
\usepackage{algorithmicx}
\usepackage{tikz}

\newcommand{\ignore}[1]{}
\def \qed {\hspace*{0pt} \hfill {\quad \vrule height 1ex width 1ex depth 0pt}
 \medskip}

\newcommand{\R}{\ensuremath{\mathbb R}}
\newcommand{\Z}{\ensuremath{\mathbb Z}}
\newcommand{\Q}{\ensuremath{\mathbb Q}}
\newcommand{\N}{\ensuremath{\mathbb N}}
\newcommand{\E}{\ensuremath{\mathbb E}}
\newcommand{\C}{\ensuremath{\mathbb C}}

\newcommand{\G}{\ensuremath{{\mathcal{G}^*}}}

\newcommand{\tl}[1]{\tilde{{#1}}}
\newcommand{\I}{\ensuremath{[0,1]}}
\newcommand{\inpr}[1]{\ensuremath{\left\langle #1 \right\rangle}}
\newcommand{\LMvarc}[1]{\ensuremath{Y_d\left(n,\frac {#1}{n}\right)}}

\newtheorem{theorem}{Theorem}[section]

\newtheorem{lemma}[theorem]{Lemma}
\newtheorem{claim}[theorem]{Claim}
\newtheorem{corollary}[theorem]{Corollary}

\newtheorem{example}[theorem]{Example}

\newtheorem{definition}[theorem]{Definition}

\newtheorem{remark}[theorem]{Remark}

\def \sign {{\rm sign}}
\def \dim {{\rm dim}}
\def \im {{\rm Im}}
\def \ker {{\rm ker}}

\def \LMc {Y_d\left(n,\frac cn\right)}
\def \ccol {c_d^{\mbox{col}}}
\def \SH {\mbox{SH}}

\begin{document}
\maketitle
\section*{Abstract}
It is well-known that the $G(n,p)$ model of random graphs undergoes a dramatic change around $p=\frac 1n$. It is here that the random graph, almost surely, contains cycles, and here it first acquires a {\em giant} (i.e., order $\Omega(n)$) connected component. Several years ago, Linial and Meshulam have introduced the $Y_d(n,p)$ model, a probability space of $n$-vertex $d$-dimensional simplicial complexes, where $Y_1(n,p)$ coincides with $G(n,p)$. Within this model we prove a natural $d$-dimensional analog of these graph theoretic phenomena. Specifically, we determine the exact threshold for the nonvanishing of the real $d$-th homology of complexes from $Y_d(n,p)$. We also compute the real Betti numbers of $Y_d(n,p)$ for $p=c/n$. Finally, we establish the emergence of giant {\em shadow} at this threshold. (For $d=1$ a giant shadow and a giant component are equivalent). Unlike the case for graphs, for $d\ge 2$ the emergence of the giant shadow is a first order phase transition.

\section{Introduction}
The systematic study of random graphs was started by Erd\H{o}s and R\'enyi in the early 1960's. It is hard to overstate the significance of random graphs in modern discrete mathematics, computer science and engineering. Since a graph can be viewed as a one-dimensional simplicial complex, it is natural to seek an analogous theory of $d$-dimensional random simplicial complexes for all $d\ge 1$. Such an analog of Erd\H{o}s and R\'enyi's $G(n,p)$ model, called $Y_d(n,p)$, was introduced in \cite{lin_mes}. A simplicial complex $Y$ in this probability space is $d$-dimensional, it has $n$ vertices and a full $(d-1)$-dimensional skeleton. Each $d$-face is placed in $Y$ independently with probability $p$. Note that $Y_1(n,p)$ is identical with $G(n,p)$.

One of the main themes in $G(n,p)$ theory is the search for {\em threshold} probabilities. If $Q$ is a monotone graph property of interest, we seek the critical probability $p=p(n)$ where a graph sampled from $G(n,p)$ has property $Q$ with probability equal to $\frac 12$. One of Erd\H{o}s and R\'enyi's main discoveries is that $p=\frac{\ln n}{n}$ is the threshold for graph connectivity. Graph connectivity can be equivalently described as the vanishing of the zeroth homology, and this suggests a $d$-dimensional counterpart. Indeed, it was shown in~\cite{lin_mes} with subsequent work in~\cite{mesh_wallach} that in $Y_d(n,p)$ the threshold for the vanishing of the $(d-1)$-th homology is $p=\frac{d\ln n}{n}$. This statement is known for all finite Abelian groups of coefficients. The same problem with integer coefficients is still not fully resolved, but see~\cite{HKP}. The threshold for the vanishing of the fundamental group of $Y_2(n,p)$ was studied in~\cite{BHK}.

Perhaps the most exciting early discovery in $G(n,p)$ theory is the so-called {\em phase transition} that occurs at $p=\frac 1n$. This is where the random graph asymptotically almost surely, i.e., with probability tending to $1$ as $n$ tends to infinity,\ acquires cycles ~\cite{book_random}. Namely for $p=o(\frac 1n)$ a $G(n,p)$ graph is asymptotically almost surely (a.a.s.) a forest. For every $0<c<1$, the probability that $G(n,\frac cn)$ is a forest approaches an explicitly computable bounded probability $0<f(c)<1$ as $n\to\infty$. Finally, for $p\ge \frac 1n$, a $G(n,p)$ graph has, a.a.s.,\ at least one cycle. Moreover, at around $p=\frac 1n$ the random $G(n,p)$ graph acquires a {\em giant component}, a connected component with $\Omega(n)$ vertices. The present work is motivated by the quest of $d$-dimensional analogs of these phenomena.

As is often the case when we consider the one vs.\ high-dimensional situations, the plot thickens here. Whereas acyclicity and collapsibility are equivalent for graphs, this is no longer the case for $d\ge 2$. Clearly, a $d$-collapsible simplicial complex has a trivial $d$-th homology, but the reverse implication does not hold in dimension $d\ge 2$. In this view, there are now two potentially separate thresholds to determine in $Y_d(n,p)$: For $d$-collapsibility and for the vanishing of the $d$-th homology. Some of these questions were answered in several papers and the present one takes the last step in this endeavour. A lower bound on the threshold for $d$-collapsibility was found in~\cite{col1} and a matching upper bound was proved in~\cite{col2}. An upper bound on the threshold for the vanishing of the $d$-th homology was found in~\cite{acyc} and here we prove a {\em matching lower bound} for the $d$-th homology over real coefficients. We conjecture that the same bound holds for all coefficient rings but this question remains open at present. Both thresholds are of the form $p=\frac cn$, but they differ quite substantially. The results allow the numerical computation of both $\ccol$ and $c_d^*$ to any desirable accuracy (See Table \ref{tbl:thrVals}).

\begin{table}
\begin{center}
\begin{tabular}{| l | c  | c | c | c |c|c|c| }
  \hline
  $d$ & \bf{2} & \textbf{3} & \textbf{4} & \textbf{5} & \textbf{10}&\textbf{100}&\textbf{1000} \\
  \hline
  $\ccol$ &
$2.455$ &
$3.089$ &
$3.509$ &
$3.822$ &
$4.749$ &
$7.555$ &
$10.175$ \\
  \hline
  $c_d^*$ &
$2.754$&
$3.907$&
$4.962$&
$5.984$&
$11-10^{-3.73}$&
$101-10^{-41.8}$&
$1001-10^{-431.7}$\\
  \hline
\end{tabular}

\caption{Values of $\ccol$ and $c_d^*$ for a selection of $d$'s.}\label{tbl:thrVals}
\end{center}
\end{table}

We turn to state the main results of this work. Note that all the asymptotic terms in this paper are with respect to the number of vertices $n$ unless stated explicitly otherwise. In addition, we only use natural logarithms. Our first main result gives the threshold for the vanishing of the $d$-th homology over $\R$, and shows that the upper bound from \cite{acyc} is tight.

\begin{theorem}
\label{thm:main}
Let $t_d^*$ be the unique root in $(0,1)$ of
\[
(d+1)(1-t_d^*)+(1+dt_d^*)\ln t_d^*=0,
\]
and let
\[c_d^* := \frac{-\ln t_d^*}{(1-t_d^*)^d}.\]
Then for every $c<c_d^*$, asymptotically almost surely, $H_d\left(\LMc;\R\right)$ is either trivial or it is generated by at most a bounded number of copies of the boundary of a $(d+1)$-simplex. 
\end{theorem}

\begin{remark}
\begin{enumerate}
\item
In Appendix~\ref{sec::techCalc} we show that $t_d^*$ and therefore $c_d^*$ are well-defined.
\item
Direct calculation shows that for large $d$, $t_d^*=e^{-(d+1)}+O_d(d^2e^{-2d})$, and $c_d^*=(d+1)(1-e^{-(d+1)})+O_d(d^3e^{-2d}).$ The threshold for $d$-collapsibility is known to be $\ccol=(1+o_d(1))\ln d$.
\item
The theorem holds also for $H_d\left(\LMc;\Z\right)$. Indeed, this is a free abelian group whose rank coincides with the dimension of the real $d$-th homology. Also, every boundary of a $(d+1)$-simplex is a $d$-cycle in the integral $d$-th homology.
\item
Let the random variable $Z$ count the copies of boundaries of a $(d+1)$-simplex in $\LMc$. It is easily verified that $Z$ is Poisson distributed with constant expectation, and in particular, $\Pr(Z=0)$ is bounded away from both zero and one. Thus the emergence of the first cycle follows a one-sided sharp transition, as does the emergence of the first cycle in a $G(n,p)$ graph.
\end{enumerate}
\end{remark}

There is an easily verifiable condition that implies that $H_d(Y;R)\neq 0$ for $Y$ a $d$-complex with a full $(d-1)$-skeleton and any ring of coefficients $R$. Namely, after all possible $d$-collapses are carried out, the remaining complex has more $d$-faces than $(d-1)$-faces that are covered by some $d$-face. By the result from \cite{acyc} and Theorem~\ref{thm:main}, for every $0\le p\le 1$ and almost all $Y\in Y_d(n,p)$, if this condition does not hold, then $H_d(Y;\R)$ is either trivial or it is generated by at most a constant number of copies of the boundary of a $(d+1)$-simplex.

We also determine the asymptotics of the Betti numbers of $\LMc$ for every $c>0$.
\begin{theorem}
\label{thm:betti}
For $c> c_d^*$, let $t_c$ be the smallest positive root of $t=e^{-c(1-t)^d}$. Then, asymptotically almost surely,
\[
\dim H_d\left(\LMc;\R\right) = {n \choose d}(1+o(1))\left(ct_c(1-t_c)^d+\frac{c}{d+1}(1-t_c)^{d+1}-(1-t_c) \right).
\]
\end{theorem}
\begin{figure}[h!]
\label{fig:caseI}
  \centering
    \includegraphics[width=0.7\textwidth]{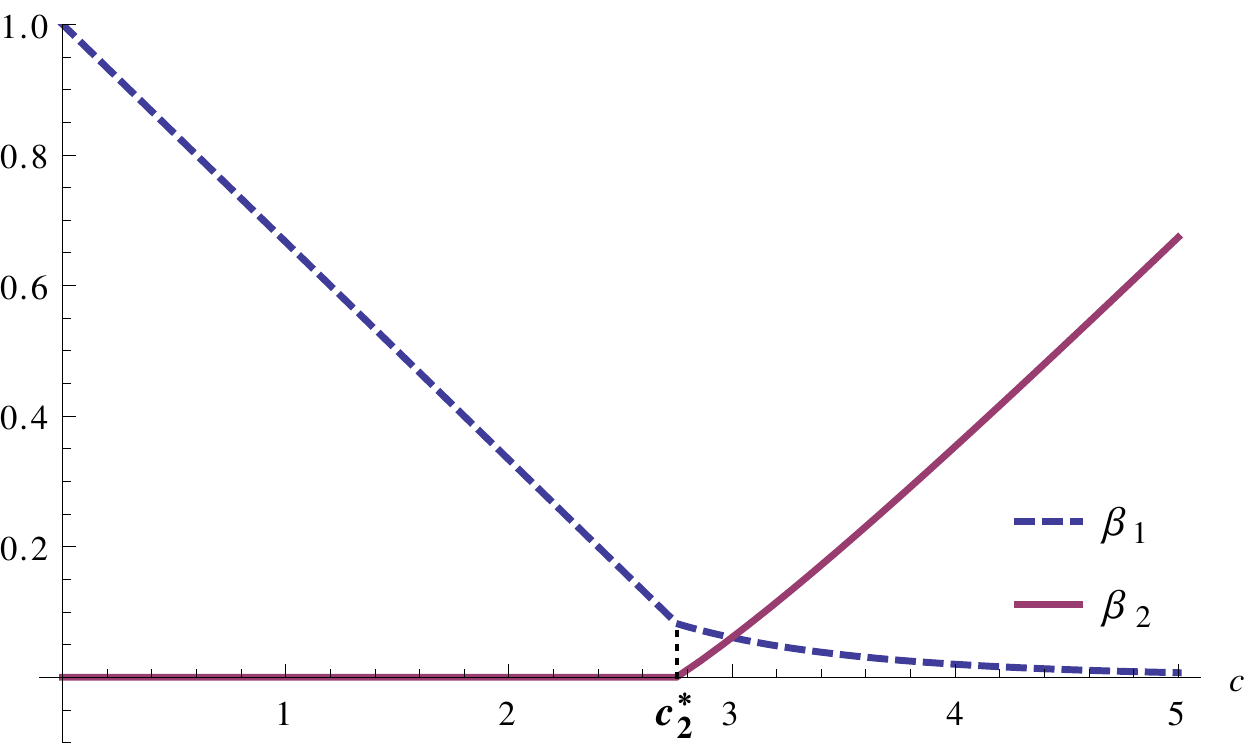}
    \caption{Illustration of Theorem \ref{thm:betti} for $d=2$. Here $\beta_i(c)=\lim_{n\to\infty}\frac{1}{{n \choose 2}}\dim H_i\left(Y_2\left(n,\frac cn\right);\R\right),~i=1,2$. Note that by Euler's formula, $\beta_1(c)-\beta_2(c)=1-\frac{c}{3}$.}
\end{figure}

There is extensive literature dealing with the emergence of the giant component in $G(n,p)$ (See, e.g., ~\cite{birth}). However, since there is no obvious high-dimensional counterpart to the notion of connected components, it is not clear how to proceed on this front. The concept of a {\em shadow}, introduced in \cite{LNPR}, suggests a way around this difficulty. The shadow of a graph $G$ is the set of those edges that are not in $G$, both vertices of which are in the same connected component of $G$. In other words, an edge belongs to $\SH(G)$ if it is not in $G$ and adding it creates a new cycle. It follows that a sparse graph has a giant component if and only if its shadow has positive density. Consequently, the giant component emerges exactly when the shadow of the evolving random graph acquires positive density. For $c>1$ the giant component of $G(n,\frac cn)$ has $((1-t_c) +o(1))\cdot n$ vertices, where $t_c$ is the root of $t=e^{-c(1-t)}$. Therefore, its shadow has density $(1-t_c)^2+o(1)$ (See Figure \ref{fig:shadowGraph}).

The above discussion suggests very naturally how to define the shadow of $Y$, an arbitrary $d$-dimensional complex with full skeleton. Note that in dimensions $d\ge 2$ the underlying coefficient field is taken into account in the definition. The $\R$-shadow of $Y$ is the following set of $d$-faces:
\[
\SH_\R(Y)=\{\sigma\notin Y:H_d(Y;\R)\text{~is a proper subspace of~} H_d(Y\cup\{\sigma\};\R)\}.  
\]
In other words, a $d$-face belongs to $\SH_\R(Y)$ if it is not in $Y$ and adding it creates a new $d$-cycle. 

The dramatic transition in the shadow's density shows a qualitative difference between the one and high-dimensional cases. Indeed, at $p=1/n$, the density of the giant component of $G(n,\frac cn)$ exhibits a continuous phase transition with discontinuous derivative. i.e. a second order phase transition. Consequently, the density of its shadow undergoes a smooth transition. In contrast, in the high-dimensional case of $d\geq 2$, the $\R$-shadow of $\LMc$ undergoes a {\em discontinuous} first-order phase transition at the criticial point $c=c_d^*$.

\begin{theorem}
\label{thm:randShadow}
Let $Y\in\LMc$ for some integer $d\ge 2$ and $c>0$ real.
\begin{enumerate}
\item If $c<c_d^*$, then a.a.s.,
$$|\SH_\R(Y)|=\Theta(n).$$
\item If $c>c_d^*$, let $t_c$ be the smallest root in $(0,1)$ of $t=e^{-c(1-t)^d}$. Then a.a.s.,
$$|\SH_\R(Y)|={{n \choose d+1}}((1-t_c)^{d+1}+o(1)).$$
\end{enumerate}
\end{theorem}

\begin{figure}[h!]
    \centering
    \subfloat[~Density of the shadow of $G\left(n,\frac cn\right)$.]{{\includegraphics[width=0.455\textwidth]{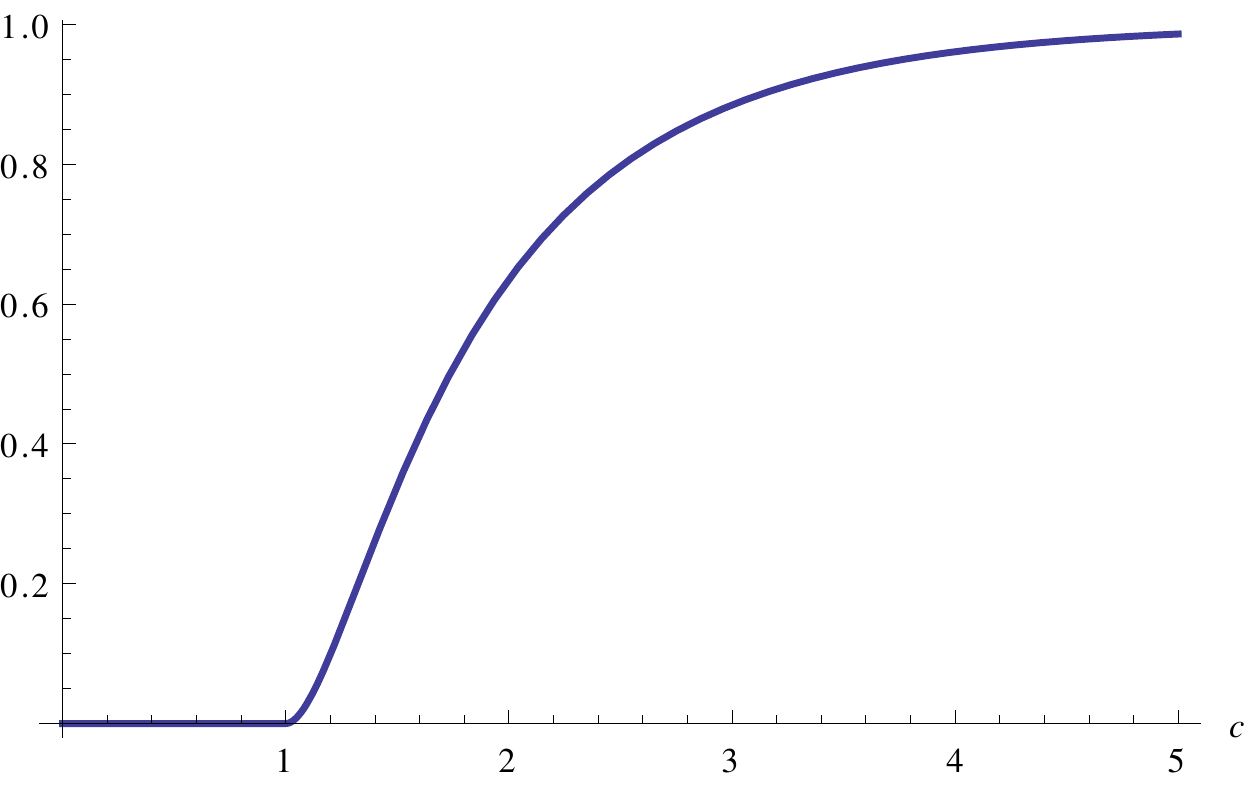} }\label{fig:shadowGraph}}%
    \qquad
    \subfloat[~Density of the $\R$-shadow of $Y_2\left(n,\frac cn\right)$.]{{\includegraphics[width=0.485\textwidth]{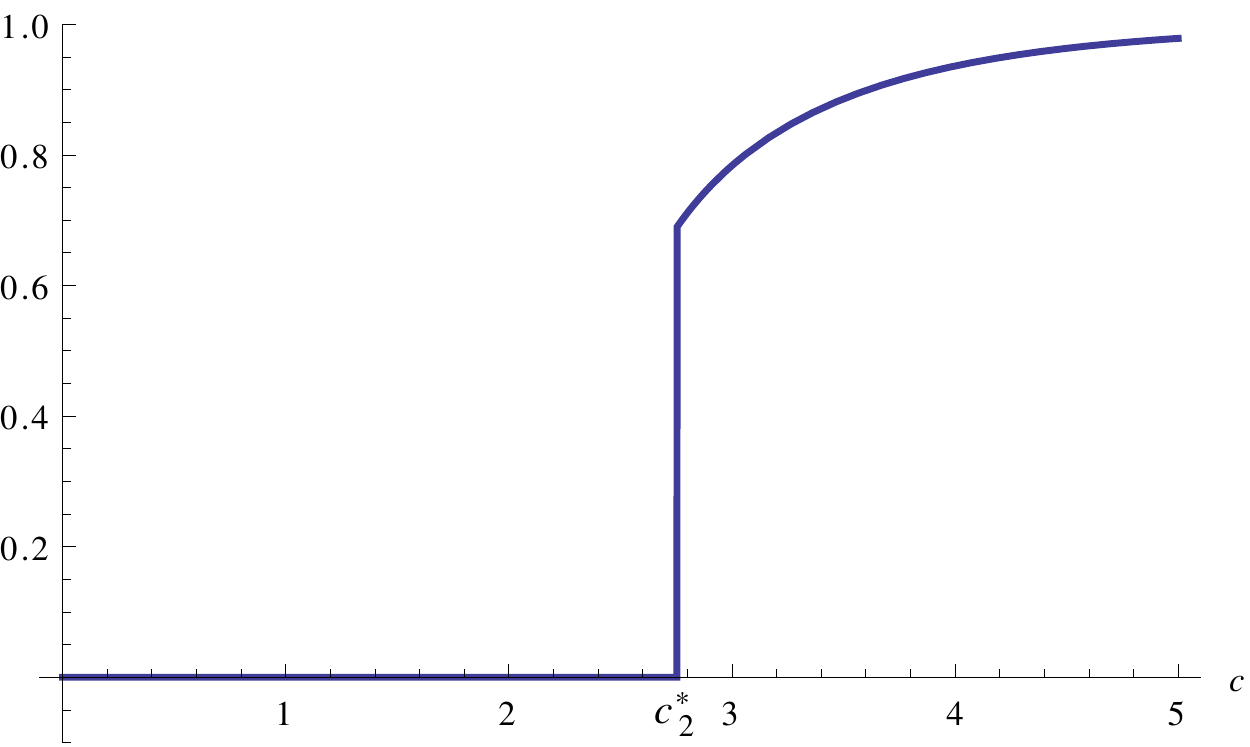} }\label{fig:shadowCompelx}}%
    \caption{Illustration of Theorem \ref{thm:randShadow} for $d=2$, and comparison to the density of the shadow of a random graph. }%
    \label{fig:shadowComparison}%
\end{figure}

It is a key idea in \cite{col1,col2,acyc} that in the range $p=\Theta(\frac 1n)$ many of the interesting properties of $Y_d(n,p)$ can be revealed by studying its {\em local} structure. In particular, this observation was essential in studying the threshold for $d$-collapsibility, and in establishing an upper bound on the threshold of the vanishing of the $d$-th homology. As explained below, this is taken here a step further with the help of the theory of local weak limits.

The root of a rooted tree is said to have {\em depth} $0$ and if $u$ is the parent of $v$, the we define $\text{depth}(v)$ to be $\text{depth}(u)+1$.

\begin{definition}
A {\em $d$-tree} is a rooted tree in which every vertex at odd depth has exactly $d$ children. A Poisson $d$-tree with parameter $c$ is a random $d$-tree in which the number of children of every vertex at even depth is a random variable with a $\text{Poi}(c)$ distribution, where all these random variables are independent.
\end{definition}

It was shown in \cite{col1}, with a slightly different terminology, that the (bipartite) incidence graph of $(d-1)$-dimensional vs.\ $d$-dimensional faces in $\LMc$ has the local structure of a Poisson $d$-tree with parameter $c$. A main challenge in this present work is to deduce {\em algebraic parameters}, such as dimensions of homology groups, from this local structure.

This naturally suggests resorting to the framework of {\em local weak convergence}, introduced by Benjamini and Schramm \cite{ben_sch} and Aldous and Steele \cite{ald_ste}. In recent years, new asymptotic results in various fields of mathematics were obtained using this approach (e.g. \cite{ald_zeta,lyo_asym}). We were particularly inspired by an impressive work of Bordenave, Lelarge and Salez \cite{diluted}, on the rank of the adjacency matrix of random graphs. They showed how to read off algebraic parameters of a sequence of combinatorial objects from its local limit. Indeed, many tools in their work turned out to be extremely useful in the study of the homology of random complexes.

Suppose $Y\in\LMc$. The group $H_d(Y;\R)$ is simply the kernel of the boundary operator $\partial_d(Y)$ of $Y$ (See Section \ref{sec:prel}). By standard linear algebra, $\dim H_d$ is expressible in terms of the dimension of the left kernel of $\partial_d(Y)$, which can be read off the spectral measure of the Laplacian operator $L_Y=\partial_d(Y)\cdot \partial_d(Y)^*$ with respect to the characteristic vector of a random $(d-1)$-face. The key idea is that this spectral measure of the Laplaican weakly converges to the spectral measure of a corresponding operator defined on the vertices of a Poisson $d$-tree, because this $d$-tree is the local weak limit of $\LMc$. Finally, the Poisson $d$-tree's spectral measure of the atom $\{0\}$ (which is the parameter required for bounding the kernel's dimension) is computed using a recursive formula exploiting the tree structure.

Our work highlights the importance of the local weak limit in the study of random simplicial complexes.
The $d$-tree is also the local weak limit of the bipartite incidence graph between vertices and hyperedges in random $(d+1)$-uniform hypergraphs in which every hyperedge is chosen independently with probability $c/{n\choose d}$ \cite{kim,hypOrient}. Collapsibility and acyclicity can be defined on hypergraphs, and these notions have been studied extensively in the contexts of random $k$-Xorsat and cuckoo hashing \cite{molloy,hypOrient,dietzfelbinger,pittel}. Surprisingly, the critical $c$'s for these hypergraph properties coincide with $\ccol$ and $c_d^*$. It is less surprising in view of the key role that the $d$-tree plays in some of these proofs. This observation illustrates the close connection between random simplicial complexes and random uniform hypergraphs at some ranges of parameters.

The rest of the paper is organized as follows. Section \ref{sec:prel} gives some necessary background material about simplicial complexes, Laplacians, operator theory and local weak convergence. In Section \ref{sec:local} we prove that the dimension of the homology of $\LMc$ can be bounded using the spectral measure of Poisson $d$-trees. In Section \ref{sec:specTree} we study the spectral measure of general and Poisson $d$-trees. In Section \ref{sec:proofs} we prove the main theorems. Concluding remarks and open questions are presented in Section \ref{sec:open}.
\section{Preliminaries}
\label{sec:prel}
\subsection{Simplicial complexes}

A simplicial complex $Y$ is a collection of subsets of its {\em vertex set} $V$ that is closed under taking subsets. Namely, if $A\in Y$ and $B\subseteq A$, then $B\in Y$ as well. Members of $Y$ are called {\em faces} or {\em simplices}. The {\em dimension} of the simplex $A\in Y$ is defined as $|A|-1$.
A $d$-dimensional simplex is also called a $d$-simplex or a $d$-face for short. The dimension $\dim(Y)$ is defined as $\max\dim(A)$ over all faces $A\in Y$. A $d$-dimensional simplicial complex is also referred to as a $d$-complex. The set of $j$-faces in $Y$ is denoted by $F_j(Y)$. For $t<\dim(Y)$, the $t$-{\em skeleton} of $Y$ is the simplicial complex that consists of all faces of dimension $\le t$ in $Y$, and $Y$ is said to have a {\em full} $t$-dimensional skeleton if its $t$-skeleton contains all the $t$-faces of $V$. The {\em degree} of a face in a complex is the number of faces of one higher dimension that contain it. Here we consider only {\em locally-finite} complexes in which every face has a finite degree.

For a face $\sigma$, the permutations on $\sigma$'s vertices are split in two
{\em orientations}, according to the permutation's sign.
The {\em boundary operator} $\partial=\partial_d$ maps an oriented
$d$-simplex $\sigma = (v_0,...,v_d)$ to the formal sum
$\sum_{i=0}^{d}(-1)^i(\sigma^i)$, where $\sigma^i=(v_0,...v_{i-1},v_{i+1},...,v_d)$ is an oriented
$(d-1)$-simplex. We fix some commutative ring $R$ and linearly extend
the boundary operator to free $R$-sums of
simplices. We denote by $\partial_d(Y)$ the $d$-dimensional boundary operator of a $d$-complex $Y$.

When $Y$ is finite, we consider the ${|F_{d-1}(Y)|}\times{|F_d(Y)|}$ matrix form of $\partial_d$ by choosing arbitrary orientations for
$(d-1)$-simplices and $d$-simplices. Note that changing the
orientation of a $d$-simplex (resp. $d-1 $-simplex) results in
multiplying the corresponding column (resp. row) by $-1$.

The $d$-th homology group $H_d(Y;R)$ (or vector space in case $R$ is a field) of a $d$-complex $Y$ is the (right) kernel of its boundary operator $\partial_d$. In this paper we work over the reals in order to use spectral methods. An element in $H_d(Y;\R)$ is called a {\em $d$-cycle}.


The {\em upper $(d-1)$-dimensional Lapalacian}, or {\em Laplacian} for short, of a complex $Y$ is the operator $L_Y=\partial_d(Y)\partial_d(Y)^*$. The kernel of the Laplacian equals to the left kernel of $\partial_d(Y)$.
For every $0\le i \le d$, the {\em $i$-th Betti number} of a complex $Y$ is defined to be the dimension of the vector quotient space $\ker \partial_i(Y) / \im\partial_{i+1}(Y)$.

A $(d-1)$-face $\tau$ in a $d$-complex $Y$ is said to be {\em exposed} if it is contained in exactly one
$d$-face $\sigma$ of $Y$. An elementary $d$-collapse on $\tau$ consists of
the removal of $\tau$ and $\sigma$ from $Y$. When the parameter $d$ is clear from the context we refer to a $d$-collapse just as collapse. We say that $Y$ is $d$-{\em collapsible} if it is possible to eliminate all the $d$-faces of $Y$ by a series of elementary $d$-collapses. A $d$-{\em core} is a $d$-complex with no exposed $(d-1)$-faces.

\subsection{Graphs of boundary operators, Laplacians and unbounded operators}
In order to use the framework of local convergence, we formulate some of the concepts and problems of interest in terms of graphs.

It is clear how to equate between matrices and weighted bipartite graphs. In particular, we can represent the boundary operator $\partial_d(Y)$ of a $d$-complex $Y$ by a bipartite graph $G_Y=(V_Y,U_Y,E_Y)$, where $V_Y=F_{d-1}(Y),~U_Y=F_d(Y)$ with edges representing inclusion among faces. In addition, edges are {\em marked} by $\pm 1$ according to the orientation. Note that every two $(d-1)$-faces can have at most one common neighbor (a $d$-face).

Accordingly, we discuss $\pm 1$ edge-marked, locally-finite (but not necessarily bounded-degree) bipartite graphs $G=(V,U,E)$, in which every two vertices $v_1,v_2\in V$ have at most one common neighbor. Associated with $G$ is an operator $L_G$ that coincides with the Lapalacian $L_Y$ for $G$ that comes from a boundary operator of some $d$-complex. Since $G$ may be infinite and have unbounded degrees, we must resort to the theory of {\em unbounded operators}~\cite{book_unbounded}. The operator $L:=L_G$ is a symmetric operator densely-defined on the subset $H$ of finitely-supported functions of the Hilbert space
$\mathcal H = \ell^2(V) = \left\{\psi:V\to\C~~~|~\sum_{v\in V}|\psi(v)|^2<\infty \right\}$.
This operator is defined by
\begin{equation}
\label{eqn:lapDef}
\inpr{Le_v,e_v} = deg(v)~~~,~~~ \inpr{Le_{v_1},e_{v_2}}= \sign(v_1,v_2),
\end{equation}

\noindent where $e_v\in\ell^2(V)$ is the characteristic function of $v\in V$. The sign function is defined via $\sign(v_1,v_2)=E(v_1,u)\cdot E(v_2,u)$, where $u$ is the unique common neighbour and $E(v,w)\in\{-1, 1\}$ is the mark on the edge $vw$. If $v_1,v_2$ have no common neighbour, then $\sign(v_1,v_2)=0$.

Note that this operator is {\em not} the Laplacian of the graph $G=(V,U,E)$. In fact, it is a marked version of the operator $A^2$ restricted to $V$, where $A$ is $G$'s adjacency operator. To avoid confusion, we will refer to it as the {\em operator of $G$}.

The densely-defined operator $L$ has a unique extension to $\mathcal H$ since it is symmetric. If this extension is a {\em self-adjoint} operator we say that $L$ is {\em essentially self-adjoint}. In such a case, the {\em spectral theorem} for self-adjoint operators implies that the action of polynomials on $L$ can be extended to every measurable function $f:\R\to\C$, uniquely defining the operator $f(L)$. In addition, associated with every function $\psi\in \mathcal H$ is a real measure $\pi_{L,\psi}$, called the {\em spectral measure of $L$ with respect to $\psi$}, which satisfies
\[
\int_\R f(x)d\pi_{L,\psi}(x) = \inpr{f(L)\psi,\psi}.
\]
In particular, $\pi_{L,\psi}$ is a probability measure if $\psi$ is a unit vector.

\begin{example} \textbf{Spectral measure in finite-dimensional spaces.}
\label{exmp:specMeasFin}
Suppose that $\dim(\mathcal H)=k<\infty$, $\psi\in\mathcal{H}$ and $L$ is a $k\times k$ Hermitian matrix. The spectral measure $\pi_{L,\psi}$ is a discrete measure supported on the spectrum of $L$, and for every eigenvalue $\lambda$
\[
\pi_{L,\psi}(\{\lambda\})=\|P_\lambda \psi \|^2,
\]
where $P_\lambda$ is the projection onto the $\lambda$-eigenspace of $L$.
In particular, if $L$ is the operator of some finite marked bipartite graph $G=(V,U,E)$, then $\dim(\ker L) = \sum_{v\in V}\pi_{L,e_v}(\{0\})$. Intuitively speaking, $\pi_{L,e_v}(\{0\})$ is the local contribution of the vertex $v$ to the kernel of $L$.
\end{example}

Spectral measures have the following continuity property. If $L,L_1,L_2,...$ are symmetric essentially self-adjoint operators densely-defined on $H$, and $L_n \psi\to L\psi$ for every vector $\psi\in H$, then the spectral measures $\pi_{L_n,\eta}$ weakly converge to $\pi_{L,\eta}$ for every $\eta\in \mathcal H$.

\subsection{Local weak convergence}

Let $G=(U,V,E)$ be a marked bipartite graph and let $v\in U\cup V$ be a vertex. A {\em flip at $v$} is an operation at which we reverse the mark on every edge incident with $v$. Two markings on $E$ are considered {\em equivalent} if one can be obtained from the other by a series of flips. Note that flips may change the operator of $G$, but if it is essentially self-adjoint, the spectral measure with respect to any characteristic function does not change.

A {\em rooted marked bipartite graph} $(G,o)$ is comprised of a marked bipartite graph $G=(V,U,E)$ and a vertex $o\in V$ - the root. An isomorphism $(G,o)\cong(G',o')$ between two such graphs is a root-preserving graph isomorphism that induces an equivalent marking on the edge sets.

Note that two rooted trees that are isomorphic as rooted graphs are also isomorphic as marked rooted graphs, since every mark pattern on the edges can be obtained by flips.

We now consider the framework of local convergence \cite{ald_ste,ben_sch} implemented with marked bipartite graphs and with the above definition of isomorphism.

Let $\G$ denote the set of all (isomorphism types of) locally-finite rooted marked bipartite graphs. For $(G,o)\in\G$ we denote by $(G,o)_k$ the radius $k$ neighborhood of $o$, i.e., the subgraph of vertices at distance $\le r$ in $G$ from the root.
There is a metric on $\G$ defined by
\[
d((G,o),(G',o')) = \inf\left\{\frac{1}{k+1}~:~(G,o)_k\cong(G',o')_k\right\}.
\]

It can be easily verified that $(\G,d)$ is a separable and complete metric space, which comes as usual equipped with its Borel $\sigma$-algebra (See \cite{ald_lyo}).

Every probability distribution $G_n=(V_n,U_n,E_n)$ on finite marked bipartite graphs induces a probability measure $\nu_n$ on $\G$ by sampling a uniform root $o\in V_n$. A probability measure $\nu$ on $\G$ is the {\em local weak limit} of $G_n$ if $\nu_{n}$ weakly converges to $\nu$. Namely, if
\[
\int_{\G} f(G,o)d\nu_n\to\int_{\G} f(G,o)d\nu
\]
for every continuous bounded function $f:\G\to\R$. Two equivalent conditions are (i) the same requirement for all bounded {\em uniformly} continuous functions $f:\G\to\R$, and (ii) $\limsup\nu_n(C)\le\nu(C)$ for every closed set $C$.

\section{Local convergence of simplicial complexes and their spectral measures}
\label{sec:local}

A basic fact about local weak convergence of graphs is that the local weak limit of the random graphs $G\left(n,\frac cn\right)$ is a Galton-Watson tree with degree distribution Poi($c$) ~\cite{dembo}.  Lemma~\ref{lem:lwc_LM} below is a high-dimensional counterpart of this fact.

Let $Y\in\LMc$ for some $d\ge 2$ and $c>0$ and let $G=(V,U,E)$ be the graph representation of the boundary operator of $Y$. Let $\nu_{d,n}$ be the probability measure on $\G$ induced by selecting a random root $o\in V=F_{d-1}(Y)$, and $\nu_{d,c}$ the probability measure on $\G$ of a Poisson $d$-tree with parameter $c$.

For every $d$-tree $T$ of finite depth $k$, denote the event $A_T=\left\{(G,o):(G,o)_k\cong T\right\}$. An essential ingredient in \cite{col1,col2,acyc} is the proof that $\nu_{d,n}(A_T)\xrightarrow{n\to\infty}\nu_{d,c}(A_T)$ for every finite $d$-tree $T$ (See, e.g., the proof of Claim 5.2 in~\cite{col1}). A straightforward calculus argument yields the following lemma.
\begin{lemma}
\label{lem:lwc_LM}
The measures $\nu_{d,n}$ weakly converges to $\nu_{d,c}$ for every integer $d\ge 2$ and real $c>0$. In other words, the local weak limit of $G_n$ is a Poisson $d$-tree with parameter $c$, where $G_n$ is the graph representing the boundary operator of $\LMc$.
\end{lemma}

We say that $(G,o) \in \G$ is {\em self-adjoint} if the corresponding operator $L_G$ is essentially self-adjoint. Note that $\mu_{G,o}:=\pi_{L_G,e_o}$, the spectral measure of $L_G$ with respect to its root is well defined since this measure depends only on the isomorphism type of $(G,o)$. More generally, a probability measure $\nu$ on $\G$ is {\em self-adjoint} if the $\nu$-measure of the set of self-adjoint members of $\G$ is $1$. A self-adjoint measure $\nu$ induces a spectral measure $\mu$ defined by
\[
\mu(E)=\int_{\G}\mu_{G,o}(E)d\nu,
\]
for every Borel set $E\subseteq\R$.

\begin{lemma}
\label{lem:continuitySpecMeas}
Suppose $(G_n,o_n)\in \G$ is a sequence of self-adjoint elements that converges to a self-adjoint element $(G,o)\in\G$. Then, the spectral measures $\mu_{G_n,o_n}$ weakly converges to  $\mu_{G,o}$.

Consequently, if a sequence of self-adjoint measures $\nu_n$ weakly converges to a self-adjoint measure $\nu$, then the induced spectral measures $\mu_n$ weakly converges to $\mu$.
\end{lemma}
\begin{proof}
Suppose $G=(V,U,E)$ and $\mathcal H = \ell^2(V)$. Let $\psi\in \mathcal H$ be a function supported on vertices of distance less than $k$ from $o$, for some integer $k$. For sufficiently large $n$, $(G,o)_k\cong(G_n,o_n)_k$, and we may as well assume that these graphs are equal. Consequently, $L_{G_n}\psi = L_G\psi$ for every sufficiently large $n$. In other words, $L_{G_n}\psi\to L_G\psi$ for every finitely supported function, and this is a sufficient condition for the weak convergence of the spectral measures with respect to the root (See Section \ref{sec:prel}).

The second item in the lemma is immediate by the definitions of weak convergence.
\end{proof}

The claim below illustrates the subtle difference between symmetric and essentially self-adjoint operators. This distinction is important because spectral measures are defined only for essentially self-adjoint operators. This question is well studied in the related context of adjacency operators of graphs ~\cite{mohar,salez}. The proof of the claim, given in Appendix \ref{sec:apxWeakConv}, is based on known methods and criteria for self-adjointness of adjacency operators ~\cite{diluted}.

\begin{claim}\label{clm:dtreeSA}
The measure $\nu_{d,c}$ is self-adjoint for every $d\ge 2$ and $c>0$.
\end{claim}

Finally, we are able to state the bound on the dimension of the kernel of the Laplacian of $\LMc$.

\begin{corollary}
\label{cor:bound_ker_tree}
Let $T$ be a Poisson $d$-tree with parameter $c$ for some integer $d\ge 2$ and $c>0$ real. Let $\mu_{T}$ be the spectral measure of the operator $L_T$ with respect to the characteristic function of the root. In addition, let $Y\in\LMc$. Then,
\[
\limsup_{n\to\infty}\frac{1}{{n\choose d}}\E_Y\left[\dim(\ker L_Y)\right] \le
\E_T\left[\mu_{T}(\{0\})\right].
\]
\end{corollary}

\begin{proof}
The measures $\nu_{d,n}$ are self-adjoint, since they are supported on finite graphs and the measure $\nu_{d,c}$ is self-adjoint by the previous claim. Consequently, the induced spectral measures $\mu_{d,n},\mu_{d,c}$ are well defined and  $\mu_{d,n}$ weakly converges to $\mu_{d,c}$. By measuring the closed set $\{0\}$ we conclude that
\[
\limsup\mu_{d,n}(\{0\}) \leq
\mu_{d,c}(\{0\})=\E_{T}\left[\mu_{T}(\{0\})\right].
\]
Let $G=(V,U,E)$ be the graph representation of $\partial_d(Y)$. 
\[
\mu_{d,n}(\{0\}) =\E_{Y}\E_{v\in V} [\pi_{L_G,e_v}(\{0\})]=
\E_Y\left[ \frac{1}{{n\choose d}} \sum_{v\in V} \pi_{L_G,e_v}(\{0\}) \right]=
\E_Y\left[ \frac{1}{{n\choose d}}\dim(\ker L_G)  \right].
\]
The first equality follows from the definition of the induced spectral measure $\mu_{d,n}$. In the next step we expand the expectation over the random vertex $v\in V$, using the fact that $|V|={n\choose d}$. For the last step, recall the remark following Example \ref{exmp:specMeasFin} regarding spectral measures in finite-dimensional spaces. 

The proof is concluded by the fact that $L_G=L_Y$, since $G$ is a graph representation of $\partial_d(Y)$.
\end{proof}

Inspired by the work of \cite{diluted}, we bound the dimension of the kernel of the Laplacian using the structure of its local weak limit. This  idea is a crucial to our work, since other approaches in the study of algebraic parameters of random graphs and hypergraphs seem inapplicable in the context of simplicial complexes.

\section{The spectral measure of a Poisson $d$-tree}
\label{sec:specTree}

Clearly the next order of things is to bound the expectation $\E_{T}\left[\mu_{T}(\{0\})\right]$. However, it is not clear how to find the spectral measure of $\{0\}$ corresponding to a given self-adjoint operator other than through a direct computation of the operator's kernel. Fortunately, for adjacency operators and Laplacians of trees, the recursive structure of trees yields simple recursion formulas on these spectral measures~\cite{resolvent,diluted}. We apply these methods to the operator $L_T$ of a $d$-tree $T$.
\subsection{A recursion formula for $d$-trees}
Given a $d$-tree $T$ with root $v$, we let $x_T:=\mu_T(\{0\})$, where $\mu_T$ is the spectral measure of the tree's operator $L_T$ with respect to the characteristic function of the root. For every vertex $v'$ of even depth, the subtree of $T$ rooted at $v'$ is the $d$-tree which contains $v'$ and its descendants.
\begin{lemma}
\label{lem:recTree}
Let $T$ be a rooted self-adjoint $d$-tree, and let $u_1,...,u_m$ be the root's children. Let $T_{j,r}~,$	 $1\le j \le m$ and $1 \le r \le d$, be the subtree of $T$ rooted at $v_{j}^{r}$, the $r$-th child of $u_j$. Then, $x_T=0$ if there exists some $1\le j \le m$ such that $x_{T_{j,1}}=...=x_{T_{j,d}}=0$. Otherwise
\[
x_T = \left(1+\sum_{j=1}^{m}\left(\sum_{r=1}^{d}x_{T_{j,r}} \right)^{-1} \right)^{-1}
\]

\end{lemma}
\begin{example}
\label{exm:demLem}
We demonstrate the recursion formula in Lemma \ref{lem:recTree} on a $d$-Tree $T$ of depth 2, that consists of a root $o$ with $m$ children, each having $d$ children. With the underlying basis $e_o,(e_{v_j^r})_{1\le j\le m,~1\le r\le d},$ $L_T$ takes the $(1+md)\times (1+md)$ matrix form
\[
L_T =
\left(\begin{matrix}
m & \textbf{j}^T   & \textbf{j}^T &\ldots & \textbf{j}^T \\
\textbf{j} & \textbf{J} & 0&\ldots & 0\\
\textbf{j} & 0& \textbf{J} & \ldots & 0\\
\vdots & \vdots & \vdots & \ddots&0\\
\textbf{j} &0&0&0&\textbf{J}
\end{matrix}\right),
\]
where $\textbf{J}$ is the $d \times d$ all-ones matrix and~ \textbf{j} one of its columns. It is easy to find a set of $m$ linearly independent columns in $L_T$, hence the dimension of $\ker L_T$ is at most $1+m(d-1)$. Consequently, the following set of vectors forms an orthonormal basis for $\ker L_T$:
\begin{enumerate}[(i)]
\item The $m(d-1)$ vectors that are obtained by a Gram-Schmidt process on the set of vectors $\{e_{v_{j}^{1}}-e_{v_j^r}\}$ where $1\le j\le m$ and $2\le r\le d$.
\item The vector $\eta:=\frac{1}{\sqrt{d^2+md}}\left(d\cdot e_o-\sum_j\sum_re_{v_j^r}\right)$.
\end{enumerate}
Since $e_o$ is orthogonal to all the basis vectors except $\eta$, we deduce from Example \ref{exmp:specMeasFin} that $x_T =\inpr {\eta,e_o}^2= \frac{d^2}{d^2+md}=\frac{d}{d+m}.$

The same conclusion follows from the recursion formula. Indeed, $x_{T_{j,r}}=1$ for every $j,r$ since $T_{j,r}$ are empty $d$-trees and their corresponding operators are null operators. By Lemma \ref{lem:recTree}, $x_T = (1+m\cdot d^{-1})^{-1}=\frac{d}{d+m}$.
\end{example}
We turn to prove the lemma in the general case.
\begin{proof}
Let us introduce some terms that we need below. We consider $T=(V,U,E)$ as a bipartite graph, and work over the Hilbert space $\mathcal H=\ell^2(V)$. Let $L$, $L_{j,r}$ denote the operators of $T$, $T_{j,r}$ resp., and let $M$ denote the operator of the subtree of depth 2 from the root (i.e., the operator from Example \ref{exm:demLem}). Consequently, $L$ admits the decomposition $L = M \oplus \tl L$, where $\tl L := \bigoplus_{j,r} L_{j,r}$. The recursion formula is derived using the {\em resolvents} of $L$ and $\tl L$. We let
$R:=R(-is;L)=(L+is\cdot I)^{-1}$ and $\tl{R} = (\tl L + is\cdot I)^{-1}$, where $s\in\R$. We denote \mbox{$A_{v_1,v_2}:=\inpr{Ae_{v_1},e_{v_2}}$}, for every operator $A$ acting on $\mathcal H$ and $v_1,v_2\in V$. By the Spectral Theorem,
$$R_{v,v} = \int_\R \frac{1}{x+is}d\mu_T(x),$$
and
$$\tl R_{v_j^r,v_j^r} = \int_\R \frac{1}{x+is}d\mu_{T_{j,r}}(x),~~1\le j\le m,~1\le r \le d.$$
It is easy to see that
(i) $\tl R e_v = \frac{1}{is}e_v$, and
(ii) $\tl R_{v_j^r,v_{j'}^{r'}}=0$ for every $(j,r)\ne (j',r')$, by the tree structure.

The recursion formula of these resolvents is proved using the Second Resolvent Identity:
$$RM\tl R =\tl R - R.$$

We compute the complex number $(RM\tl R)_{v,v} = (\tl R-  R)_{v,v}$ in two ways. On the one hand, since $M$ is supported only on $v$ and the $v_j^r$'s, and $\tl R e_v = \frac{1}{is}e_v$, it holds that
\[
(RM\tl R)_{v,v} = R_{v,v}M_{v,v}\tl R_{v,v} + \sum_{j=1}^{m}\sum_{r=1}^{d} R_{v,v_j^r}M_{v_j^r,v}\tl R_{v,v}.
\]
Using the concrete structure of the operator $M$ (See Example \ref{exm:demLem}), this can be restated as
\[
(RM\tl R)_{v,v} = \frac{1}{is}\left(m R_{v,v} + \sum_{j=1}^{m}\sum_{r=1}^{d}R_{v,v_j^r}\right).
\]
On the other hand,
\[
(\tl R - R)_{v,v} = \frac{1}{is}-R_{v,v}.
\]
A comparison of these two terms yields:
\begin{equation}
\label{eqn:firstRecEqn}
isR_{v,v}+\sum_{j=1}^{m}\left(R_{v,v}+\sum_{r=1}^{d}R_{v,v_j^r}\right)=1.
\end{equation}

Similarly, we compute the complex number $(RM\tl R)_{v,v_j^r}=(\tl R-R)_{v,v_j^r}$ for every $j,r$.
\[
(RM\tl R)_{v,v_j^r}=\tl R_{v_j^r,v_j^r}\left(R_{v,v}+\sum_{r'=1}^{d}R_{v,v_j^{r'}}\right).
\]
Consequently,
\[
\sum_{r=1}^{d}(RM\tl R)_{v,v_j^r}=\left(\sum_{r=1}^{d}\tl R_{v_j^r,v_j^r}\right)\left(R_{v,v}+\sum_{r=1}^{d}R_{v,v_j^r}\right).
\]
On the other hand,
\[
\sum_{r=1}^{d}(\tl R-R)_{v,v_j^r}=-\sum_{r=1}^{d}R_{v,v_j^r}.
\]
By comparing these last two terms,
\begin{equation}
\label{eqn:secondRecEqn}
\sum_{r=1}^{d}R_{v,v_j^r} = -R_{v,v}\left(\frac{\sum_{r=1}^{d}\tl R_{v_j^r,v_j^r}}{1+\sum_{r=1}^{d}\tl R_{v_j^r,v_j^r}} \right).
\end{equation}
(Below we explain why the denominator ${1+\sum_{r=1}^{d}\tl R_{v_j^r,v_j^r}}$ does not vanish).

By combining Equations $(\ref{eqn:firstRecEqn}),(\ref{eqn:secondRecEqn})$, we obtain a recursion formula of the resolvents.
\begin{equation}
\label{eqn:SteiRec}
R_{v,v}\left(is +\sum_{j=1}^{m}\left(\frac{1}{1+\sum_{r=1}^{d}\tl R_{v_j^r,v_j^r}}\right)\right)=1.
\end{equation}

We next turn to derive the recursion formula on $x_T$ from the recursion of the resolvents.

Let $h_T(s)=is\int_{\R}\frac{1}{x+is}d\mu_{T}(x)=isR_{v,v}$. Then
\[
h_T(s)=is\int\frac{x-is}{x^2+s^2}d\mu_T(x)=
\int\frac{s^2}{x^2+s^2}d\mu_T(x)+i\int\frac{xs}{x^2+s^2}d\mu_T(x)
\]

Note that the pointwise limit of $\frac{xs}{x^2+s^2}$ as $s\to 0$ is the zero function. Also the pointwise limit of $\frac{s^2}{x^2+s^2}$ as $s\to 0$ is the Kronecker delta function $\delta_0$. Since both these families of real functions are bounded, the dominant convergence theorem implies that $h_T(s) \xrightarrow{s\to 0} \mu_T(\{0\})=x_T$. We can similarly define $h_{T_{j,r}}(s)=is\tl R_{v_j^r,v_j^r}$, and by the same argument,  $h_{T_{j,r}}(s) \xrightarrow{s\to 0} x_{T_{j,r}}$. Equation (\ref{eqn:SteiRec}) takes the form:
\[
h_T(s)\left( 1+  \sum_{\j=1}^{m}\left( is+\sum_{r=1}^{d}h_{T_{j,r}}(s) \right)^{-1}\right) = 1.
\]
The proof is concluded by letting $s\to 0$.

Note that $is+\sum_{r=1}^{d}h_{T_{j,r}}(s)$ does not vanish, since the real part of  $h_{T_{j,r}}(s)$ is strictly positive. This also explains why the denominator in (\ref{eqn:secondRecEqn}) does not vanish.
\end{proof}

\subsection{Solving the recursion for Poisson $d$-trees}
We will now deduce a concrete bound on the spectral measure of a Poisson $d$-tree using the recursion formula. The proof of Lemma \ref{lem::recPoiTree} below follows ideas from \cite{diluted}.
\begin{lemma}
\label{lem::recPoiTree}
Let $T$ be a rooted Poisson $d$-tree with parameter $c$, and $\mu_T$ be the spectral measure with respect to its root. Then,
\[
\E[\mu_T(\{0\}) ]\le \max\left\{t+ct(1-t)^d-\frac{c}{d+1}\left(1-(1-t)^{d+1}\right)~\mid~ t\in\I,~t=e^{-c(1-t)^d}\right\}
\]
\end{lemma}
\begin{remark}
Due to the condition $t=e^{-c(1-t)^d}$, this maximum is always over a finite set. In fact, there are at most three possible values of $t$, see Appendix~\ref{sec::techCalc} for details.
\end{remark}
\begin{proof}
Let $\mathcal D$ denote the distribution of $\mu_T(\{0\})\in\I$, where $T$ is a Poisson $d$-tree with parameter $c$. We next define a real-valued random variable $X$ and denote its distribution by $\mathcal D'$. To define $X$ we sample first an integer $m\sim \mbox{Poi}(c)$ and $X_{j,r}\sim\mathcal{D}$ i.i.d.\ for every $1\le j\le m$ and $1\le r\le d$. Given these samples, $X$ takes the value $0$ if there exists some $j$ for which $X_{{j,1}}=\ldots=X_{{j,d}}=0$. Otherwise
\[
X=\left(1+\sum_{j=1}^{m}\left(\sum_{r=1}^{d}X_{j,r} \right)^{-1} \right)^{-1}.
\]

The recursion formula of Lemma $\ref{lem:recTree}$ implies the distributional equation $\mathcal D=\mathcal D'$, since
every vertex at depth two in a Poisson $d$-tree is the root of a Poisson $d$-tree.

The definitions of $\mathcal D,\mathcal D'$ yields the following equation for the probability $t:=\Pr(X>0)$:
\begin{equation}\label{eqn:p_eqn}
t = \sum_{m=0}^{\infty}\frac{e^{-c}c^m}{m!}(1-(1-t)^d)^m=e^{-c(1-t)^d}.
\end{equation}
Let $S,S_1,S_2,\ldots$ be random variables whose distribution is that of a sum of $d$ i.i.d.\ $\mathcal D$-distributed variables.

\begin{align}
\E[X]~=~&\E\left[\frac{\textbf{1}_{\{\forall j\in [m],\;S_j>0\}}}{ 1+  \sum_{j=1}^{m}S_j^{-1} }\right]\nonumber\\
=~&\E\left[\textbf{1}_{\{\forall j,\;S_j>0\}}\left(1- \frac{\sum_{j=1}^{m}S_j^{-1}}{ 1+  \sum_{j=1}^{m}S_j^{-1} }\right)\right]\nonumber\\
=~&t-\E\left[ \sum_{i=1}^{m} \frac{S_i^{-1}\cdot \textbf{1}_{\{\forall j \;S_j>0\}}}{1+S_i^{-1}+\sum_{j\ne i}S_j^{-1} } \right]\nonumber\\
\label{eqn_arry:2}
=~&t-\E_m\left[m\cdot\E\left[ \frac{S^{-1}\cdot \textbf{1}_{\{S>0;\;\forall j \; S_j>0\}}}{1+S^{-1}+\sum_{j=1}^{m-1}S_j^{-1} } \right]\right]\\
\label{eqn_arry:3}
=~&t-c\cdot\E\left[ \frac{S^{-1}\cdot \textbf{1}_{\{S>0;\;\forall j \; S_j>0\}}}{1+S^{-1}+\sum_{j=1}^{m}S_j^{-1} } \right]\\
\label{eqn_arry:4}
=~&t-c\cdot\E\left[\frac{X}{X+S}\cdot\textbf{1}_{\{S>0,\; X>0\}}\right]\\
\label{eqn_arry:5}
=~&t-c\left( \sum_{i=1}^{d}{d \choose i}t^{i+1}(1-t)^{d-i}\frac{1}{i+1}\right).
\end{align}

Equation (\ref{eqn_arry:2}) is obtained by linearity of expectation, since the $m$ random variables $\frac{S_i^{-1}\cdot \textbf{1}_{\{\forall j \;S_j>0\}}}{1+S_i^{-1}+\sum_{j\ne i}S_j^{-1} },$ \mbox{$i=1,...,m,$} are identically distributed.

To derive Equation (\ref{eqn_arry:3}) recall that $\E[m\cdot \varphi(m-1)] = c\cdot\E[\varphi(m)]$, provided that $m\sim\mbox{Poi}(c)$. This holds for every function $\varphi:\N\to\R$.

We pass to Equation (\ref{eqn_arry:4}) by multiplying both the numerator and the denominator by $S~\Big/ \left(1+\sum_{j=1}^{m}S_j^{-1}\right)$, using the fact that $X\sim\mathcal D'$.

To see why Equation (\ref{eqn_arry:5}) holds, note the following. By linearity of expectation, if $Z,Z_1...,Z_i$ are i.i.d.\ {\em positive} random variables, then $\E[Z~\Big/ (Z+Z_1+...+Z_i)]=1/(i+1)$. In our case, the probability that $X$ and exactly $i$ out of the $d$ summands in $S$ are positive equals to ${d \choose i}t^{i+1}(1-t)^{d-i}.$

The proof is completed with the following straightforward calculation:
\[
t-c\left( \sum_{i=1}^{d}{d \choose i}t^{i+1}(1-t)^{d-i}\frac{1}{i+1}\right)=
t+ct(1-t)^d-\frac{c}{d+1}\left(1-(1-t)^{d+1}\right).
\]
%
%
\end{proof}

We conclude this section by restating the bound in Lemma \ref{lem::recPoiTree} in concrete terms. The proof, which uses only basic calculus, is in Appendix \ref{sec::techCalc}.
\begin{lemma}
\label{lem:explBounds}
Recall the definition of $c_d^*~$  from Theorem \ref{thm:main}. Then, the maximum of $$t + ct(1-t)^d-\frac{c}{d+1}\left(1-(1-t)^{d+1}\right),~~\mbox{such that}~~~t=e^{-c(1-t)^d},$$ is attained at:
\begin{enumerate}
\item $t=1$, for $c<c_d^*$. In particular, the maximum equals $1-\frac{c}{d+1}$.
\item The smallest root $t_c$ in $(0,1)$ of the equation $t=e^{-c(1-t)^d}$ for $c\ge c_d^*$.
\end{enumerate}
\end{lemma}

\section{Proofs of the main theorems}
\label{sec:proofs}
\subsection{From expectation to high probability - proof of Theorem \ref{thm:main}}
We start with the range $c<c_d^*$. Let $Y$ be an $n$-vertex $d$-complex. We apply the rank-nullity theorem from linear algebra to $\partial_d(Y)$ and its adjoint to conclude that $$\dim H_d(Y;\R)-\dim(\ker L_Y)  = |F_{d}(Y)|-|F_{d-1}(Y)|.$$ For $Y\in\LMc$ this becomes
\[
\E[\dim H_d(Y;\R)]-\E[\dim(\ker L_Y)]=\frac cn{n\choose d+1}-{n \choose d}.
\]
By the results from the previous sections, and in particular the first item of Lemma \ref{lem:explBounds} we deduce that
\begin{align*}
\limsup \frac{1}{{n \choose d}}\E[\dim H_d(Y;\R)]~=~&\limsup\left( \frac{1}{{n \choose d}}\E[\dim(\ker L_Y)]+\frac {c}{d+1}\left(1-\frac dn\right)-1\right)\\
\le~&\E_T[\mu_T(\{0\})]-\left(1-\frac{c}{d+1}\right)\\
=~&0.
\end{align*}

\begin{proof}[Proof of Theorem \ref{thm:main}.]
Now we complete the proof of Theorem \ref{thm:main}, by proving a high probability statement. To this end we recall the following a.a.s.\ characterization of minimal cores in $\LMc$ (Theorem 4.1 from \cite{col1}). Namely, for every $c>0$ a.a.s.\ every minimal core in $\LMc$ is either the boundary of $(d+1)$-simplex, or it has cardinality at least $\delta n^d$, where $\delta>0$ depends only on $c$. Since every $d$-cycle is a core we conclude:
\begin{lemma}
\label{lem:noSmallCycles}
For every $c>0$ a.a.s.\ every $d$-cycle of $\LMc$ that is not the boundary of $(d+1)$-simplex is {\em big}, i.e., it has at least $\delta n^d$ $d$-faces. Here $\delta>0$ depends only on $c$.
\end{lemma}

To finish the proof of Theorem \ref{thm:main}, all we need, then, is to rule out the existence of big cycles.
Let $Y_0\in\LMvarc{c'}$ for some $c<c'<c_d^*$. As we showed $\E[\dim H_d(Y_0;\R)] = o(n^d)$, and so, by Markov inequality, a.a.s.\ $\dim H_d(Y_0;\R)= o(n^d)$. Sample uniformly at random $|F_d(Y_0)|$ numbers from $\I$ and let $k$ be the number of these samples that are $< 1-c/c'$. Clearly, a.a.s.\ $k=\Theta(n^d)$. Define the $d$-complexes $Y_0\supset Y_1 \supset Y_2 \supset ... \supset Y_k$, where $Y_{i+1}$ results by removing a random $d$-face $\sigma_i$ from $Y_i$ for $i=0,1,\ldots, k-1$. Clearly, $Y_k\in\LMc$.

If $Y_i$ contains a big $d$-cycle, then with probability bounded away from zero, the random $d$-face $\sigma_{i+1}$ is in it, in which case, $\dim H_d(Y_{i+1};\R)=\dim H_d(Y_{i};\R)-1.$

It follows that if $Y_k$ has a big cycle, then $\{\dim H_d(Y_i;\R)\}_{i=1}^k$ is a random sequence of $\Omega(n^d)$ nonnegative integers, that starts with a value of $o(n^d)$, and has a constant probability of dropping by $1$ at each step. A contradiction.

\end{proof}

\subsection{Betti numbers of $\LMc$ - proof of Theorem \ref{thm:betti}}

We now deal with the range $c>c_d^*$. Let $t_c$ be the smallest root of $t=e^{-c(1-t)^d}$ in $(0,1)$. Let $Y\in\LMc$. The argument of the previous paragraph and the second item of Lemma \ref{lem:explBounds} imply that
\[
\limsup \frac{1}{{n \choose d}}\E[\dim H_d(Y;\R)]\le
\E_T[\mu_T(\{0\})]-\left(1-\frac{c}{d+1}\right)=
\]
\begin{equation}
\label{eqn:bettiTerm}
=ct_c(1-t_c)^d+\frac{c}{d+1}(1-t_c)^{d+1}-(1-t_c)=:g_d(c).
\end{equation}
This upper bound matches a lower bound that is derived by analyzing the following process on a $\LMc$ complex. We first carry out a large but constant number of comprehensive collapse steps, and then we remove all uncovered $(d-1)$-faces from the remaining complex. As shown in \cite{acyc}, the expected difference between the number of $d$-faces and $(d-1)$-faces in the remaining complex is ${n \choose d}(1+o(1))g_d(c)$. A simple linear algebraic consideration yields that $g_d(c)$ is also a lower bound for $\liminf \frac{1}{{n \choose d}}\E[\dim H_d(Y;\R)]$. Consequently,
\[
\E[\dim H_d(Y;\R)] = {n\choose d}(1+o(1))g_d(c).
\]
The fact that a.a.s.\ $\dim H_d(Y;\R) = {n\choose d}(1+o(1))g_d(c)$ is shown by the following version of the Azuma inequality from \cite{azuma}.\\
\begin{claim}\label{azuma}
Let $\Phi:\{0,1\}^m\to\R$ be a function with the property that $|\Phi(z)-\Phi(z')|\le 1$ whenever $z$ and $z'$ differ at exactly one coordinate. If $Z_1,...,Z_m$ be independent indicator random variables, then for every $r>0$
\[
\Pr[|\Phi(Z_1,...,Z_m)-\E[\Phi(Z_1,...,Z_m)]|\ge r]\le 2e^{-2r^2/m}.
\]
\end{claim}

We apply this inequality with $m={n \choose d+1}$ and $r=n^{(d+2)/2}$. Fix some ordering $\sigma_1,...,\sigma_{{n\choose d+1}}$ of all $d$-faces, and let $Z_i$ be the indicator of the event $\sigma_i\in Y$. The function $\Phi = \dim H_d(Y;\R)$ clearly satisfies the assumption of Claim~\ref{azuma}. It follows that
\[
\Pr\left[\left|\dim H_d(Y;\R)-{n\choose d}(1+o(1))g_d(c)\right| \ge n^{(d+2)/2}\right] \leq 2e^{-\Omega(n)}\to 0,
\]
which concludes the proof.

\subsection{Shadows of random complexes - proof of Theorem \ref{thm:randShadow}}
The first item in the theorem regarding the range $c<c_d^*$ is easy. We first consider $d$-faces that are in the shadow because their addition completes the boundary of a $(d+1)$-simplex. But a second moment calculation shows that a.a.s.\ there are $\Theta(n)$ sets of $d+2$ vertices in $Y$ that span all but one of the $d$-faces in the boundary of a $(d+1)$-simplex. The rest of the proof proceeds in reverse along the argument of Lemma~\ref{lem:noSmallCycles}:
Assume for contradiction that $|\SH_\R(Y)|\gg n$. This implies that for every $c<c'<c_d^*$ the following holds with probability bounded away from zero: The $d$-th homology of $\LMvarc{c'}$ contains a $d$-cycle that is not the boundary of $(d+1)$-simplex. But this contradicts Theorem \ref{thm:main}.

We turn to the range $c>c_d^*$. Recall that $\frac{1}{{n\choose d}} \E[\dim H_d(Y;\R)]=g_d(c)+o(1)$, where $$g_d(c)=ct_c(1-t_c)^d+\frac{c}{d+1}(1-t_c)^{d+1}-(1-t_c).$$
We need the following technical claim, which is proved in Appendix \ref{sec::techCalc}.
\begin{claim}
\label{clm:techDiff}
For every $c>c_d^*$, $g_d(c)$ is differentiable and $~g_d'(c)=\frac{1}{d+1}(1-t_c)^{d+1}$.
\end{claim}

We now derive a lower bound on the density of the $\R$-shadow. Fix some $c>c_d^*$, and assume toward contradiction that the event
\begin{equation}\label{eqn:smallShadow}
\frac{1}{{n\choose d+1}}\left|\SH_\R\left(\LMc\right)\right|<(1-t_c)^{d+1}-\alpha.
\end{equation}

\noindent holds with probability bounded from zero, for some $\alpha>0$. As in the proof of Theorem \ref{thm:main}, we employ a $d$-dimensional analog of the so-called {\em evolution of random graphs}. We fix some small $\varepsilon>0$, and start with the $n$-vertex complex $Y_0\in\LMvarc{c-\varepsilon}$. For $i=0,...,m-1$ we obtain the complex $Y_{i+1}$ by adding a random $d$-face $\sigma_i\notin Y_i$ to ${Y_i}$. The parameter $m$ is sampled randomly so as to guarantee that $Y_m\sim \LMc$. A standard concentration argument implies that a.s.\ $m=\frac{\varepsilon}{d+1}{n \choose d}(1+o(1))$ with an exponentially small probability of error.

Clearly, $\dim H_d(Y_{i+1};\R)=\dim H_d(Y_{i};\R)+Z_i$, where $Z_i$ is the indicator random variable of the event $\sigma_i\in \SH_\R(Y_i)$. We condition on the event that (i) $Y_m$ satisfies relation (\ref{eqn:smallShadow}) and (ii) $m=\frac{\varepsilon}{d+1}{n \choose d}(1+o(1))$. By assumption and by a previous comment, with probability bounded from zero both conditions are satisfied. Since $\SH_\R(Y_i)\subseteq\SH_\R(Y_{i+1})\cup\{\sigma_i\}$ for $0\le i\le m$, the densities of $\SH_\R(Y_i)$ are nondecreasing (up to $-o(1)$ terms), and in particular, $\SH_\R(Y_i)$ has density $\le (1-t_c)^{d+1}-\alpha+o(1)$. Consequently, under the above mentioned conditioning, the random variable $\sum Z_i$ is stochastically bounded from above by a binomial random variable with $\frac{\varepsilon}{d+1}{n \choose d}(1+o(1))$ experiments and success probability of $(1-t_c)^{d+1}-\alpha+o(1)$. By Chernoff's inequality (See ~\cite{book_random}), this binomial variable is a.a.s.\ bounded from above by $\frac{\varepsilon}{d+1}{n \choose d}((1-t_c)^{d+1}-\alpha/2)$.

In conclusion, our assumption implies that with probability bounded from zero,
$$\dim H_d(Y_m;\R)-\dim H_d(Y_0;\R) = \sum Z_i < \frac{\varepsilon}{d+1}{n \choose d}((1-t_c)^{d+1}-\alpha/2).$$
On the other hand, by Theorem \ref{thm:betti}, a.a.s.,
$$\dim H_d(Y_m;\R)-\dim H_d(Y_0;\R)=\binom{n}{d}(g_d(c)-g_d(c-\varepsilon)+o(1)).$$ Since $g_d'(c)=\frac{1}{d+1}(1-t_c)^{d+1}$, this yields a contradiction when $\varepsilon$ is sufficiently small.

We next establish a matching upper bound. Fix some $c>c_d^*$ and assume toward contradiction that
\begin{equation}\label{eqn:bigShadow}
\frac{1}{{n\choose d+1}}\left|\SH_\R\left(\LMc\right)\right|>(1-t_c)^{d+1}+\alpha.
\end{equation}
\noindent holds with probability bounded from zero, for some $\alpha>0$.
Fix some $\varepsilon>0$, and consider, as above, an increasing sequence of random complexes $Y_0,...,Y_m$ with $Y_0\in\LMc$ and $Y_m\in\LMvarc{c+\varepsilon}$, where $Y_{i+1}$ is created by adding a random $d$-face $\sigma_i\notin Y_i$ to ${Y_i}$. Note that $\dim H_d(Y_m;\R)-\dim H_d(Y_0;\R) \ge |\{\sigma_0,...,\sigma_{m-1}\}\cap\SH_\R(Y_0)|$. Indeed, every $\sigma_i$ in the shadow of $Y_0$ contributes a $d$-cycle that contains aside from itself only $d$-faces from $Y_0$. Such a cycle is therefore linearly independent of the other $d$-cycles.

By assumption, with probability bounded from zero, $F:=\{\sigma_0,...,\sigma_{m-1}\}$ is a random set of $d$-faces of size $\frac{\varepsilon}{d+1}{n \choose d}(1+o(1))$ and $\SH_\R(Y_0)$ has density as in (\ref{eqn:bigShadow}). Consequently, $$|F\cap\SH_\R(Y_0)|>\frac{\varepsilon}{d+1}{n \choose d}((1-t_c)^{d+1}+\alpha/2)$$ holds with probability bounded away from zero, which yields a contradiction similarly to the previous case.

\section{Concluding remarks and open questions}
\label{sec:open}

\begin{itemize}
\item We work throughout with $\R$ as the underlying coefficient ring. We suspect that the threshold for the vanishing of the $d$-th homology does not depend on the coefficient ring. Presumably the best place to start these investigations is $R=\Z_2$.

\item Here we view the phase transition in $G=G(n,\frac cn)$ at $c=1$ as reflected in the growth of the shadow. More traditionally, one considers instead the growth of the graph's connected components. This information can be conveniently read off the left kernel of the graph's vertices-edges boundary matrix. In analogy, one may investigate the structure of the left kernel of $\partial_d(Y)$, and the $(d-1)$-th cohomology group in $Y=\LMc$.
Here we found this group's dimension for every $c$, but much remains unknown about its structure.

\item Kalai ~\cite{kalai} introduced $\Q$-acyclic complexes (or hypertrees) as high-dimensional analogs of trees. These are sets of $d$-faces whose corresponding columns in $\partial_d$ constitute a basis for the column space of this matrix. Grimmett's Lemma \cite{grimmettLemma} determines the local weak limit of a random tree. In view of the role played by local weak limits in the study of $\LMc$ we ask: What is the local weak limit of random hypertrees?

\item Many questions on a uniformly drawn random hypertree remain open: What is the probability that it is $d$-collapsible? Its integral $(d-1)$-th homology is a finite group. How is its size distributed?

Consider the following random process: First pick a random ordering $\sigma_1,...,\sigma_{n\choose d+1}$ of the $d$-faces. Now create a sequence of complexes that starts with a full $(d-1)$-skeleton and no $d$-faces.  At each step we add to the complex the next $d$-face according to $\sigma$ that does not form a $d$-cycle when added. This process ends after ${n-1 \choose d}$ steps with a random hypertree $T$. Equivalently, $T$ is a min-weight hypertree where $d$-faces are assigned random weights.

Fix some $\ccol<c<c_d^*$ and let $i\sim\mbox{Bin}\left({n \choose d+1},\frac cn \right)$. The $d$-complex $Y$ that has a full $(d-1)$-skeleton and the set $\{\sigma_1,...\sigma_i\}$ as its $d$-faces is precisely $\LMc$. By Theorem \ref{thm:main} and \cite{col2} we know that $Y$ is a.a.s.\ not $d$-collapsible and it does not contain $d$-cycles except, possibly, a constant number $k$ of boundaries of $(d+1)$-simplices. Remove one $d$-face from each of these $k$ boundaries of $(d+1)$-simplices to obtain a complex $Y'\subseteq Y$ that is acyclic and not $d$-collapsible. Note that $Y'$ is a subcomplex of $T$. Concretely, its $d$-faces are exactly the first $i-k$ $d$-faces that are placed in $T$. Consequently, {\em $T$ is a.a.s.\ not $d$-collapsible}.

Quite a few basic questions concerning such random hypertrees are open: What is their local weak limit? How large is the integral $(d-1)$-th homology?

\item We know that for $c>c_d^*$, a $d$-cycle in $\LMc$ is either the boundary of a $(d+1)$-simplex or it has $\Omega(n^d)$ faces, but many structural issues remain unknown. The following phenomena are observed in numerical experiments, but there is still no proof or refutation. (i) All such {\em big} cycles contain all $n$ vertices. (ii) Consider an inclusion-minimal $d$-cycle $C$. Every $(d-1)$-face that is contained in a $d$-face of $C$ must clearly have degree $\ge 2$ in $C$, equality being attained by closed manifolds. On the other hand, by a simple degree argument, the average degree of such $(d-1)$-faces is $\le {d+1}$. In our experiments, this average degree in the $d$-cycles in $\LMc$ is consistently close to $d+1$.

\item We have determined here the density of the $\R$-shadow of $\LMc$ for every $c>0$. It would be of interest to give a more detailed description of its combinatorial structure.

\end{itemize}


\appendix
\section{Proof of Claim \ref{clm:dtreeSA} }
\label{sec:apxWeakConv}
Let $T=(V,U,E)$ be a Poisson $d$-tree with parameter $c$, considered as a bipartite graph with root $o\in V$. Namely, $V$ (resp. $U$) is the set of vertices of even (odd) depth. Let $L:=L_T$ be the operator of $T$ and $L^*$ its adjoint. We follow a method used in \cite{diluted} for adjacency operators of Galton-Watson trees with finite first moment. By a characterization of essentially self-adjoint operators~\cite{book_unbounded}, it is sufficient to show that $\ker(L^*\pm i)=0$.

A {\em trim} $R$ of $T$ is a finite subtree rooted at $o$, all leaves of which belong to $V$.
The {\em fan-out} of $R$ is the maximal number of children that a leaf in $R$ has as a vertex in $T$.

\begin{claim}
For every $c>0$ and $d\ge 2$, there exists a constant $k=k(d,c)$ such that almost surely (a.s.),\ a Poisson $d$-tree $T$ with parameter $c$ has a trim of fan-out at most $k$.
\end{claim}

\begin{proof}
Let $k$ be large enough so that $\sum_{m>k}\frac{e^{-c} c^m}{m!}dm < 1$. Consider the random $d$-tree in which the number of children of each vertex $v$ at even depth is determined as follows: We independently sample from Poisson distribution $Poi(c)$. If the sample is $k$ or less, this is the number of $v$'s children. Otherwise, $v$ is childless, i.e., a leaf. In this branching process, the expected number of grandchildren of a vertex $v$ is $\sum_{m>k}\frac{e^{-c} c^m}{m!}dm < 1$, and therefore it is a.s.\ finite. If we do this process on $T$ while it is being generated, we obtain the desired trim.
\end{proof}

\begin{corollary}
A Poisson $d$-tree $T$ can a.s.\ be covered by trims of fan-out $\le k=k(d, c)$. I.e., there exists a sequence $R_1\subset R_2 \subset\ldots$ of trims such that (i) Every vertex of $T$ is in some $R_j$, and (ii) For every $j$ every leaf in $R_j$ has at most $k$ children in $T$.
\end{corollary}

\begin{proof}
Let $R_1$ be the trim from the previous claim. The general case is proved inductively by applying the same claim to the subtrees of $T$ rooted at the leaves of previous trims.
\end{proof}

We now turn to use the above criterion and show that a $d$-tree that is covered by trims of bounded fan-out is essentially self-adjoint. Suppose that $L^*\psi = -i\psi$ for some $\psi \in \ell^2(V)$ in the domain of $L^*$ (the case $+i$ is very similar). Namely, for every finitely supported function $\varphi$,
$\inpr{L\varphi,\psi}=\inpr{\varphi,-i\psi}$.

As usual $N(x)$ denotes the neighbor set of vertex $x$ in $T$. We define a function $F(v\to u)$ on pairs of neighbors $v\in V, u \in U$ as follows
\[
F(v\to u):= \im\left[\psi(v)  \sum_{v'\in N(u)}\overline{\psi(v')}
\right].
\]

We occasionally think of $F$ as a {\em flow}. Concretely we note two key properties of $F$: (i) The total flow {\em into} every vertex $u\in U$ is zero, and (ii) The total {\em out} of every vertex $v\in V$ is $|\psi(v)|^2$. Indeed, for every $u\in U$,
\[
\sum_{v\in N(u)} F(v\to u)=\im\left[  \sum_{v\in N(u)}\psi(v)\overline{\sum_{v\in N(u)}\psi(v)}\right]=0.
\]
Also, for every $v\in V$,
\[
\sum_{u\in N(v) }F(v\to u)=
\sum_{u\in N(v)}\im\left[\psi(v)\sum_{v'\in N(u)}\overline{\psi(v')}
\right]=
\im\left[\psi(v)\inpr{\sum_{u\in N(v)}\sum_{v'\in N(u)}e_{v'},\psi}
\right]=\]\[
\im\left[\psi(v)\inpr{Le_v,\psi}\right]=
\im\left[\psi(v)\inpr{e_v,-i\psi}\right]=|\psi(v)|^2.
\]

Two notations that we need are: The set of leaves in a trim $R$ is denoted $\Delta(R)$. Also for $v\in V$ we denote by $C(v)\subseteq U$ the set of $v$'s children in $T$. It follows that for every trim $R$ of $T$,
\[
\sum_{v\in R\cap V}|\psi(v)|^2 = \sum_{v\in R\cap V}\sum_{u\in N(v)} F(v\to u)=\sum_{v\in \Delta(R)}\sum_{u\in C(v)}F(v\to u),
\]
Note that $|F(v\to u)|\le |\psi(v)|\sum_{v'\in N(u)}|\psi(v')|$. Consequently, by applying Cauchy-Schwartz twice,
\[
\sum_{v\in R\cap V}|\psi(v)|^2 \le
\sum_{v\in \Delta(R)}|\psi(v)|\sum_{u\in C(v)}\sum_{v'\in N(u)}|\psi(v')|\le
\]
\[
\left(\sum_{v\in \Delta(R)}|\psi(v)|^2 \right)^{1/2}
\left(\sum_{v\in \Delta(R)}\left(\sum_{u\in C(v)}\sum_{v'\in N(u)}|\psi(v')|\right)^2\right)^{1/2}\le
\]
\[
\left(\sum_{v\in \Delta(R)}|\psi(v)|^2 \right)^{1/2}
\left(\sum_{v\in \Delta(R)}|C(v)|(d+1)\sum_{u\in C(v)}\sum_{v'\in N(u)}|\psi(v')|^2\right)^{1/2}.
\]
By considering the sequence of trims with bounded fan-out from the previous claim, we obtain
\[
\sum_{v\in R_j\cap V}|\psi(v)|^2 \leq
\sqrt{k(d+1)}\left(\sum_{v\in \Delta(R_j)}|\psi(v)|^2 \right)^{1/2}
\left(\sum_{v\in \Delta(R_j)}\sum_{u\in C(v)}\sum_{v'\in N(u)}|\psi(v')|^2\right)^{1/2}.
\]
Clearly, $\sum_{v\in R_j\cap V}|\psi(v)|^2\to\|\psi\|^2$ when $j\to\infty$. On the other hand, let $t_j$ be the depth of the set $\Delta(R_j)$. The right hand side is bounded by a constant times the sum of $|\psi(v)|^2$ over vertices at depth at least $t_j$. This is arbitrarily small, since $t_j\to \infty$ and $\|\psi\|<\infty$.

\section{Some technical proofs }
\label{sec::techCalc}

In this appendix we prove Lemma \ref{lem:explBounds}, a technical claim that appears implicitly in Theorem \ref{thm:main} and Claim \ref{clm:techDiff}. Figure~\ref{fig:caseI} can help in following the general description of the proof that we now give. We seek to maximize $f(t):=t+ct(1-t)^d-\frac{c}{d+1}\left(1-(1-t)^{d+1}\right)$ subject to $t=e^{-c(1-t)^d}$. The equation $t=e^{-c(1-t)^d}$ has the root $t=1$, and for $1>t>0$ it takes the form $\psi(t)=c$, where
\[
\psi(t)=\frac{-\ln t}{(1-t)^d}.
\]
As we show, there is some $1>t_\psi>0$ such that $\psi$ is decreasing in $(0,t_\psi)$ and increasing in $(t_\psi,1)$. Therefore the equation $\psi(t)=c$ has at most two roots in $(0,1)$, and we only need to find the largest number among $f(1)=1-\frac{c}{d+1}$ and at most two other values of $f$.
We then observe that
\begin{equation}\label{f_varphi}
\text{If~} t=e^{-c(1-t)^d}, \text{~then~} f(t)>f(1) \text{~iff~} \varphi(t)<0
\end{equation}
where
\[
\varphi(t)=(d+1)(1-t)+(1+dt)\ln t.
\]
As implicitly stated in Theorem \ref{thm:main}, and as we soon show, there is some $1>t_d^*>0$ such that $\varphi(t)$ is negative in $(0,t_d^*)$ and positive in $(t_d^*,1)$. Consequently, the relevant maximum of $f$ occurs at $t=1$ unless the equation $\psi(t)=c$ has a root in $(0,t_d^*)$. As we show, exactly one such a root, namely $t=t_c$, exists exactly when $c>c_d^*:=\psi(t_d^*)$.

We turn to fill in the details. Clearly, $\psi(t) \to \infty$ when $t\to 1^-$ (since $d\ge 2$), or $t \to 0^+$. Also, $$\psi' = -\frac{1-t+dt\ln t}{t(1-t)^{d-1}}.$$
Consequently, $\psi$ has a unique local extremum $1>t_\psi>0$ which is a minimum. We recall from \cite{col1,col2} that this minimum $\ccol=\psi(t_\psi)$ is the threshold for $d$-collapsibility. It follows that in $(0,1)$ the equation $\psi(t)=c$ has
(i) No roots when $0<c<\ccol$, (ii) A single root $t=t_\psi$ when $c=\ccol$, and (iii) Two roots $t=t_1(c)$, $t=t_2(c)$, satisfying $t_1(c) < t_\psi < t_2(c) < 1$ when $c>\ccol$.

To prove the claim in Equation~(\ref{f_varphi}), note that if $t=e^{-c(1-t)^d}$, then
\[
\varphi(t)=(d+1)(1-t)-(1+dt)c(1-t)^d=\]\[
(d+1)\left(1-t-ct(1-t)^d-\frac{c}{d+1}(1-t)^{d+1}\right)=
(d+1)\left(1-\frac{c}{d+1}-f(t)\right).
\]

It is easily verified that $\varphi(t) \to -\infty$ when $t\to 0^+$. In addition, the Taylor expansion of $\varphi(t)$ at $t=1$ yields that $\varphi(t)=\frac{d-1}{2}(1-t)^2 +O((1-t)^3)$. Hence $\varphi(t)\searrow 0$ when $t\to 1^-$, since $d\ge 2$. Also,
$$\varphi '= \frac{1}{t}(1-t+dt\ln t)=-\frac{1}{(1-t)^{d-1}}\psi'$$
and since $\psi'$ vanishes exactly once at $(0,1)$, at $t_\psi$, it follows that
$\varphi$ has a unique extremum in $(0,1)$, at $t=t_\psi$, which is clearly a maximum.

Our analysis of $\varphi$ yields that as implicitly assumed in the statement of Theorem \ref{thm:main}, $\varphi$ vanishes exactly once in $(0,1)$, at a point that we call $t_d^*$. Moreover, $\varphi$ is negative in $(0,t_d^*)$ and positive in $(t_d^*,1)$, and $t_d^*<t_\psi$, where $\varphi$ takes its unique maximum value.

To complete the proof, note that $\psi$ is decreasing in $(0,t_\psi)$ and $t_\psi>t_d^*$, so the equation $\psi(t)=c$ has a root in $(0,t_d^*)$ iff $c>\psi(t_d^*)=c_d^*$. By definition this root is $t=t_c$.
\qed
\begin{figure}[h!]
  \centering
    \includegraphics[width=0.8\textwidth]{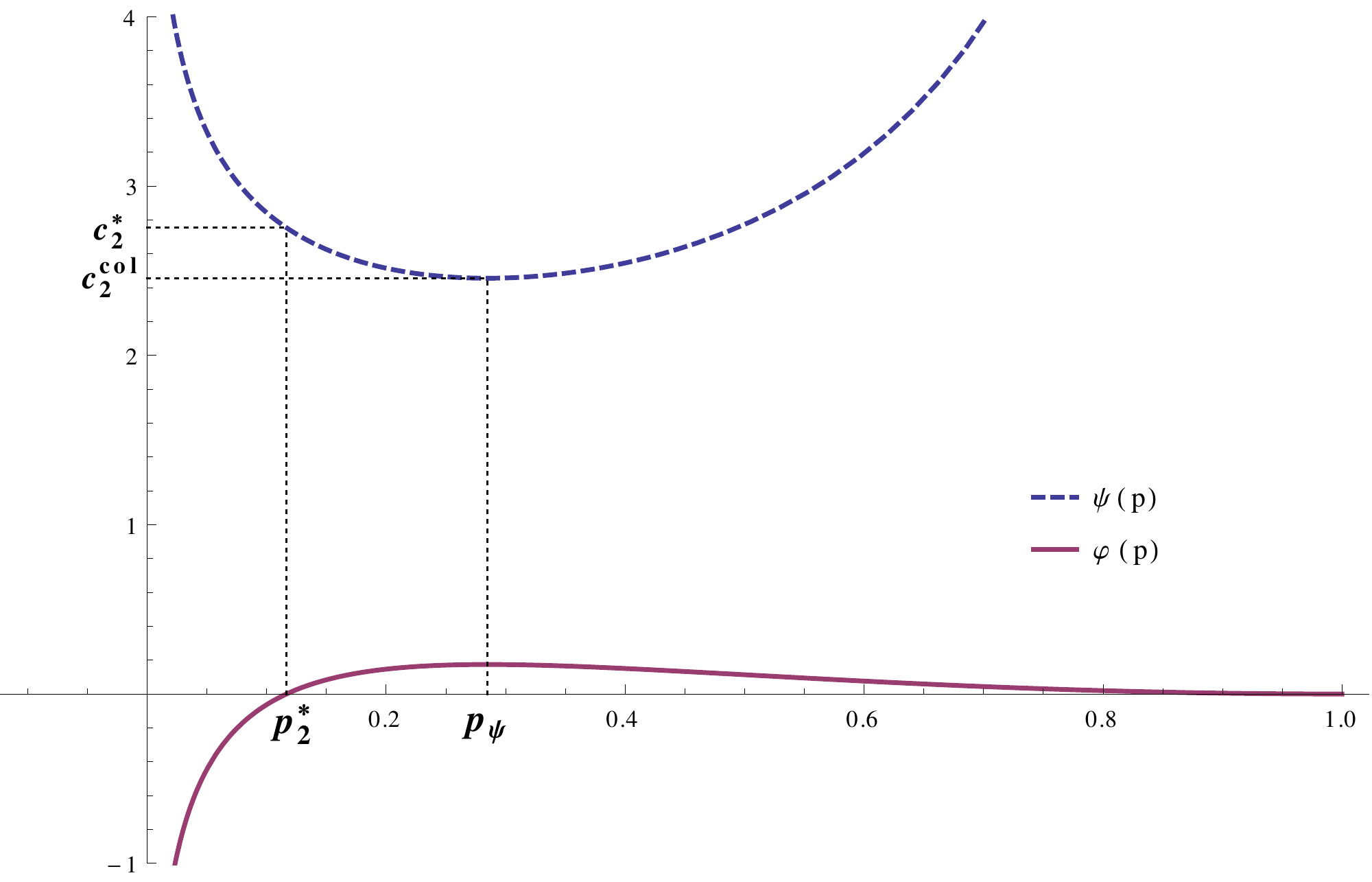}
    \caption{Illustration of the proof of Lemma \ref{lem:explBounds} for $d=2$. This is the qualitative picture for every $d\ge 2$.}
    \label{fig:caseI}
\end{figure}

We conclude by proving Claim \ref{clm:techDiff}.
\begin{proof}[Proof of Claim \ref{clm:techDiff}]
Consider $t_c$ as a function of $c$, defined implicitly as the smaller root of $\psi(t)=c$. We denote derivatives w.r.t.\ $c$ by $'$ and find that
\[
t_c' = \frac{1}{\psi'(t_c)}=-\frac{t_c(1-t_c)^{d+1}}{1-t_c+dt_c\ln t_c}.
\]
By straightforward calculation,
\[
g_d'(c)=t_c(1-t_c)^d+\frac{1}{d+1}(1-t_c)^{d+1}+t_c'(1-dct_c(1-t_c)^{d-1}).
\]
Since $c(1-t_c)^d=-\ln t_c$, this can be restated as
\[
g_d'(c)=t_c(1-t_c)^d+\frac{1}{d+1}(1-t_c)^{d+1}+t_c'\cdot\frac{1-t_c+dt_c\ln t_c}{1-t_c}=\frac{1}{d+1}(1-t_c)^{d+1}.
\]
\end{proof}

\end{document}